\documentclass[11pt]{amsart}

\usepackage[utf8]{inputenc} 
\usepackage[english]{babel}
\usepackage[left=1.5in, right=1.5in, top=1.5in, bottom=1.5in]{geometry}
\usepackage{amssymb}        
\usepackage{amsthm}         
\usepackage{mathtools}      
\usepackage{physics}        
\usepackage{xcolor}         
\usepackage[
    hyperfootnotes=false      
]{hyperref}                 
\usepackage[capitalize]{cleveref}       
\usepackage{enumitem}       
\usepackage{chngcntr}       
\usepackage{xcolor,colortbl}

\definecolor{greenrb}{rgb}{0.2,0.6,0.2}
\definecolor{rred}{rgb}{0.7,0,0.1}

\newcommand{\R}{\mathbb{R}}
\newcommand{\T}{\mathbb{T}}

\newcommand{\Z}{\mathbb{Z}}
\newcommand{\C}{\mathbb{C}}
\newcommand{\PP}{\mathcal{P}}
\DeclareMathOperator*{\argmax}{arg\,max}

\DeclareMathOperator*{\re}{Re}
\DeclareMathOperator*{\im}{Im}

\def\ol{\overline}

\def\be{\begin{equation}}
\def\ee{\end{equation}}

\definecolor{linkblue}{HTML}{00356B}
\definecolor{linkgold}{HTML}{DA9100}
\definecolor{linkred}{RGB}{159,  29, 53}
\hypersetup{
	colorlinks = true,
	linkcolor = linkred,
	citecolor = linkred,
	urlcolor = linkblue
}

\newtheorem{theorem}{Theorem}[section]

\newtheorem{proposition}[theorem]{Proposition}
\theoremstyle{definition}
\newtheorem{definition}{Definition}
\theoremstyle{definition}
\newtheorem*{remark}{Remark}

\def\titletext{Dynamic Transitions of the Swift-Hohenberg Equation with Third-Order dispersion}
\title[Transitions of S-H Equation with Third-Order dispersion]{\MakeUppercase{\titletext}}
\author[Kevin Li]{Kevin Li}
\address{Department of Mathematics \\ Yale University \\ New Haven, CT 06510, USA}
\email[Kevin Li]{k.li@yale.edu}

\date{\today}

\begin{document}

\maketitle

\begin{abstract}
    The Swift-Hohenberg equation is ubiquitous in the study of bistable dynamics. In this paper, we study the dynamic transitions of the Swift-Hohenberg equation with a third-order dispersion term in one spacial dimension with a periodic boundary condition. As a control parameter crosses a critical value, the trivial stable equilibrium solution will lose its stability, and undergoes a dynamic transition to a new physical state, described by a local attractor. The main result of this paper is to fully characterize the type and detailed structure of the transition using dynamic transition theory \cite{phase_transition_dynamics}. In particular, employing techniques from center manifold theory, we reduce this infinite dimensional problem to a finite one since the space on which the exchange of stability occurs is finite dimensional. The problem then reduces to analysis of single or double Hopf bifurcations, and we completely classify the possible phase changes depending on the dispersion for every spacial period. 
\end{abstract}
\setcounter{tocdepth}{1}
\tableofcontents

\section{Introduction}
The standard Swift-Hohenberg equation was introduced to describe the onset of Rayleigh-Benard convection, which considers a horizontal layer of viscous fluid heated from below. It is given by 
\[\frac{\partial u}{\partial t} = -(1 + \Delta)^2 u + \lambda u - u^3,\]
and is crucial to the study of non-equilibrium physic due to the natural formation of convection patterns as the fluid is heated. As the control parameter $\lambda$ increases, the equation exhibits distinct pattern forming behavior \cite{2d_swift, pattern_formation}. Spacial and temporal patterns occur when systems transition from a basic stable state and bifurcates to nontrivial attractors when the control parameter crosses a critical value. The phase transition dynamics of the standard Swift-Hohenberg equation is well understood in one and two spacial dimensions \cite{2d_swift}. Beyond hydrodynamics, this fourth-order equation has a plethora of applications in bistable dynamics. Of particular interest to us is when a ring cavity made of optical fibers is driven by a beam. When the beam is operating near the resonant frequency of the cavity, the dynamics exhibited by the the cavity is modelled by the Swift-Hohenberg equation with a third-order dispersion term \cite{phys}.

In this paper, we fully characterize the phase transition dynamics of the Swift-Hohenberg equation with a third-order dispersion in one spacial dimension, given by 
\begin{equation} \label{SH_system}
\begin{cases}
u_t = \lambda u - (1 + \partial_x^2)^2 u + \sigma \partial_x^3 u + b u^2 - u^3 \quad \text{for } x \in [0, \ell], \, t \ge 0 \\ u(t, 0) = u(t, \ell) \quad \forall t \ge 0.
\end{cases}
\end{equation}
Let $\T$ denote the torus of measure $\ell$. Then we can consider 
\begin{equation*}\label{operator}
    L_\lambda = \lambda - (1 + \partial_x^2)^2 + \sigma \partial_x^3
\end{equation*}
as an operator from $H^4(\T) \to L^2(\T)$. Note that the nonlinear mapping $G(u, \lambda) = bu^2 - u^3$ can be expanded as $G(u, \lambda) = G_2(u, \lambda) + o(\lVert u \rVert^2_{H^4(\T)})$ where $G_2(u_1, u_2, \lambda) = b u_1 u_2$ is a one parameter family of bilinear mappings. Since we have the embedding $H^4(\T) \hookrightarrow C^3(\T)$, then \cref{SH_system} with the boundary condition can be recasted as 
\be 
u_t = L_\lambda u + G(u, \lambda).
\ee
Small perturbation of the control parameter, $\lambda$, about a critical value leads to significantly different dynamics near the trivial equilibrium point. 

To analyze the exchange of stability that occurs across a critical value of $\lambda$, we use the framework of dynamic transition theory set forth by Ma and Wang \cite{phase_transition_dynamics}. In equilibrium thermodynamics, the standard characterization of phase transition is through the Ehrenfest classification, where transitions are classified based on the regularity across a critical value in the control parameter with respect to some thermodynamic potential. For example, as the temperature crosses the $0^\circ$C threshold for a sample of water, the first order derivative of the Gibbs potential function is discontinuous at $T = 0^\circ$C, so it is a first order transition. This largely hides the dynamic properties of the system in state space. For general dynamical systems, it is natural to analyze how an attractor loses stability as the control parameter crosses a critical value, in relations to local attractors. 

The paper will be organized as follows. We start with some preliminaries, in which we give an overview of general dynamic transition theory, as well as a general description of the dynamic transitions that accompany a Hopf and double Hopf bifurcation. After the preliminaries, the main theorem will be given by \cref{MAIN}, which will fully characterize the dynamic transitions of the Swift-Hohenberg equation with third-order dispersion in one dimension. In particular, we will show in \cref{section:dynamic_transitions} that for $\ell \le 2 \pi /\sqrt{2}$, the transition and dynamics are completely determined by whether $b$ is non-zero in \cref{SH_system}. For $\ell \le 2 \pi/\sqrt{2}$, we will either end up with a simple Hopf bifucation or a double Hopf bifuction, the behaviors of which will be fully characterized in \cref{th:single_hopf} and \cref{th:double_hopf}.

\section{Preliminaries}
We start with a general dynamical system. Let $X_1$ and $X$ be two Banach spaces such that $X_1 \subset X$ is a compact inclusion. Consider the nonlinear evolution equation 
\begin{equation}\label{dynamic_system} 
\frac{\partial u}{\partial t} = L_\lambda u + G(u, \lambda), \quad u(0) = u_0 
\end{equation}
and $L_\lambda$ is a one parameter family of linear completely continuous fields depending continuously on $\lambda$, i.e.
\be \label{completey_continuous_cond}
\begin{array}{ll}
    L_\lambda = -A + B_\lambda & \quad \text{is a sectorial operator} \\
    A: X_1 \to X & \quad  \text{is a linear homeomorphism} \\
    B_\lambda: X_1 \to X & \quad  \text{linear compact operators parameterized by } \lambda.
\end{array}
\ee
In this case, we can define fractional power operators $L^\alpha_\lambda$ with domain $X_{\alpha} = D(L^\alpha_\lambda)$. We further assume that $G( \cdot, \lambda) : X_\alpha \to X$ is a $C^r$ bounded mapping for some $0 \le \alpha < 1$ and $r \ge 1$, and depends continuously on $\lambda$, and for all $\lambda \in \R$,
\be
G(u,\lambda) = o(\lVert u \rVert_{X_{\alpha}})
\ee 
This implies that $u = 0$ is a trivial equilibrium point. Clearly there is no loss of generality in this assumption since we may just shift the system by a constant.

\begin{definition} \cite{phase_transition_dynamics}
The system given by \cref{dynamic_system} satisfying \cref{completey_continuous_cond} undergoes a dynamic transition at $\lambda = \lambda_0$ if the following conditions hold:
\begin{enumerate}
    \item if $\lambda < \lambda_0$, \cref{dynamic_system} is locally asymptotically stable at $u = 0$,
    
    \item if $\lambda > \lambda_0$, $\text{codim} (\Gamma_\lambda) \ge 1$ where $\Gamma_\lambda$ is the stable manifold of \cref{dynamic_system} about $u = 0$ for $\lambda$, and there exists a neighborhood $U \subset X$ of $u = 0$ independent of $\lambda$ such that for any $u_0 \in U \setminus \Gamma_\lambda$, the solution $u_\lambda$ with initial vaue $u_0$ satisfies
    \[\limsup_{t \to \infty} \lVert u_{\lambda}(t, u_0) \rVert_X \ge \delta(\lambda) > 0, \qquad \lim_{\lambda \to \lambda_0^+} \delta(\lambda) \ge 0.\]
\end{enumerate}
\end{definition}
This amounts to saying that a dynamic transition happens when the equilibrium point at $u = 0$ loses stability across $\lambda_0$. The simplest way this occurs is if eigenvalues of $L_\lambda$ crosses the imaginary axis as $\lambda$ crosses $\lambda_0$. If such an exchange of stability occurs, then we can classify the dynamic transition into three categories. 
\begin{theorem}\cite{phase_transition_dynamics}
Let the eigenvalues of $L_\lambda$ be given by $\{\beta_i(\lambda) : i = 1, 2, \dots\}$, and suppose
\begin{equation}\label{PES}
    \begin{array}{ll}
    \re(\beta_i) \begin{cases} < 0 & \text{if } \lambda < \lambda_0 \\ = 0 & \text{if } \lambda = \lambda_0 \\ > 0 & \text{if } \lambda > \lambda_0 \end{cases} & \qquad 1 \le i \le m\\
    \re(\beta_i) < 0 & \qquad i > m,
\end{array}
\end{equation}
then \cref{dynamic_system} undergoes a dynamic transition at $\lambda = \lambda_0$, and is one of the following three types:
\begin{enumerate}[label = (\roman*)]
    \item there exists a neighborhood $U \subset X$ of $u = 0$ and dense subsets $\widetilde{U}_\lambda$ so that for any $u_0 \in \widetilde{U}_\lambda$, solution $u_\lambda(t, u_\lambda)$ satisifes
    \[\lim_{\lambda \to \lambda_0^+} \limsup_{t \to \infty} \lVert u_{\lambda}(t, u_0) \rVert = 0\]
    
    \item there exists a neighborhood $U \subset X$ of $u = 0$ and dense subset $\widetilde{U}_\lambda$ and an $\epsilon > 0$ so that for any $\lambda_0 < \lambda < \lambda_0 + \epsilon$ and any $u_0 \in \widetilde{U}_\lambda$, solution $u_\lambda(t, u_\lambda)$ satisfies
    \[\limsup_{t \to \infty} \lVert u_\lambda(t, u_0) \rVert \ge \delta > 0\]
    for some fixed $\delta$ independent of $\lambda$. 
    
    \item there exists a neighborhood $U \subset X$ of $u = 0$ such that for any $\lambda_0 < \lambda < \lambda_0 + \epsilon$ for some $\epsilon > 0$, we have a decomposition $U^1_\lambda$ and $U^2_\lambda$ so that $\overline{U} = \overline{U}^1_\lambda \cup \overline{U}^2_\lambda$ and $U^1_\lambda \cap U^2_\lambda = \varnothing$, and
    \[\begin{array}{ll} \lim_{\lambda \to \lambda_0^+} \limsup_{t \to \infty} \lVert u_{\lambda}(t, u_0) \rVert = 0 & \quad u_0 \in U^1_\lambda \\ \limsup_{t \to \infty} \lVert u_\lambda(t, u_0) \rVert \ge \delta > 0 &\quad u_0 \in U^2_\lambda \end{array} \]
    for some fixed $\delta$ independent of $\lambda$. 
\end{enumerate}
\end{theorem}
These transitions are called \textit{continuous, catastrophic,} and \textit{mixed} transitions respectively, and we call $\beta_i$ with $1 \le i \le m$ the critical eigenvalues. A proof of the theorem is given by Ma-Wang \cite{phase_transition_dynamics}. Physically, a phase transition occurs when a system leaves a basic local attractor to another, called the transition states, as a system parameter $\lambda$ crosses the threshold value $\lambda_0$. In a continuous transition, the transition states attracts a neighborhood of the basic state. In a catastrophic transition, the transition states are local attractors away from the basic state. Note that this classification has a natural relation to the Ehrenfest classification for equilibrium phase transitions. In particular, catastrophic implies first order since stability is lost abruptly and jumps away from the basic state. Continuous transition corresponds directly to second order transition. For the complete relation between the two classifications, see \cite{phase_transition_dynamics}.

If the system satisfies the conditions given in (\ref{PES}), it suffices to analyze the dynamics on the finite dimensional center manifold tangent to the subspace spanned by eigenvectors $\beta_i$ with $1 \le i \le m$, since the space orthogonal to the center subspace is stable. Hence any solution with initial value near $0$ will be attracted towards the center manifold. Furthermore, the center manifold is tangent to the center subspace, so the dynamics on the center manifold near the trivial equilibrium point is equivalent to the dynamics of the solution projected onto the center subspace. In particular, we have the following two useful theorems that extend the Hopf bifurcation theorem. 

\begin{theorem}\label{th:single_hopf} \cite{double_hopf,ma2005bifurcation,phase_transition_dynamics,appendix}
Consider a system in the form of \cref{dynamic_system} that satisfies (\ref{completey_continuous_cond}). Assume that there are two critical eigenvalues that are complex simple, $\beta_1$ and $\ol \beta_1$, so that assumption (\ref{PES}) holds. Further assume that the solutions to the system projected onto the critical eigenvector $\phi_1$ for sufficiently small initial value and $\lambda$ near $\lambda_0$ satisfies
\begin{equation}\label{reduced_ode}
    \frac{d z}{dt} = \beta_1(\lambda) z + P(\lambda) z |z|^2 + o(|z|^3)
\end{equation}
where $z$ is the amplitude of projection of the solution and $P(\lambda)$ is continuously differentiable near $\lambda = \lambda_0$. $P = P(\lambda_0) \in \C$ is a constant called the transition number. Then the transition of the system is completely characterized by the dynamics at $\lambda = \lambda_0$. In particular, we have the following characterization:
\begin{enumerate}[label = (\roman*)]
    \item If $\re(P) < 0$, the system undergoes a continuous transition to a local attractor $\sigma_{\lambda}$ homological to $S^1$, and the basic steady state solution $u = 0$ bifurcates to a stable periodic solution $u_\lambda$ on $\lambda > \lambda_0$.
    \item If $\re(P) > 0$, the system undergoes a catastrophic transition, and the basic steady state solution $u = 0$ bifurcates to an unstable periodic solution $u_\lambda$ on $\lambda < \lambda_0$.
    \item $u_\lambda$ in both (1) and (2) have the approximation 
    \begin{equation}
        u_\lambda(x, t) = 2\left( \frac{-\re \beta_1(\lambda) }{\re P} \right)^{1/2} \re \left(e^{i \omega t} \phi_1 \right) + o(\re \beta)^{1/2}
    \end{equation}
    where
    \begin{equation*}
        \omega = \im \beta_1 - \im P \frac{\re \beta_1}{\re P}
    \end{equation*}
\end{enumerate}
\end{theorem}

\begin{proof}
Making the substitution $z = \rho(t) e^{i \gamma(t)}$ where $\rho$ and $\gamma$ are real-valued functions and $\rho(t) > 0$, we find that the real part of the differential equation is given by 
\begin{equation}\label{real_part_sub}
    \rho'(t) = \re \beta_1(\lambda)\rho(t) + \re P(\lambda) \rho(t)^3 + o(|z|^3).
\end{equation}
At $\lambda = \lambda_0$, $\re \beta_1(\lambda_0) = 0$, so clearly $\rho = 0$ is an asymptotically stable equilibrium point if $\re P < 0$, and an unstable equilibrium point if $\re P > 0$. Thus the transition is continuous if $\re P < 0$ and catastrophic if $\re P > 0$. By the Hopf bifurcation theorem, the stable equilibrium must bifurcate to a periodic orbit for $\lambda > \lambda_0$ if $\re P > 0$. From \cref{real_part_sub}, we see that $z = \frac{-\Re \beta_1(\lambda)}{\re P}^{1/2} e^{i \gamma(t)} + o(\Re \beta)^{1/2}$. Taking the imaginary part in the substitution, we find 
\[\gamma'(t) = \im \beta + \im P \rho^2(t),\]
from which it follows that $\gamma(t) = \omega t$. 
\end{proof}

\begin{theorem}\label{th:double_hopf}
Consider a system in the form of \cref{dynamic_system} that satisfies (\ref{completey_continuous_cond}). Assume that there are 2 pairs of conjugate critical eigenvalues, $\beta_1$, $\beta_2$, $\ol \beta_1$, and $\ol \beta_2$, and that assumption (\ref{PES}) holds. Further assume that the solution of the system projected onto the critical eigenvectors $\phi_1$ and $\phi_2$ for sufficiently small initial value and $\lambda$ near $\lambda_0$ satisfies
\begin{gather}
    \frac{dz_1}{dt} = \beta_1(\lambda) z_1 + z_1(A(\lambda)|z_1|^2 + B(\lambda)|z_2|^2) + o((|z_1| + |z_2|)^3) \label{4crossclass1}\\
    \frac{dz_2}{dt} = \beta_2(\lambda) z_2 + z_2(C(\lambda)|z_1|^2 + D(\lambda)|z_2|^2) + o((|z_1| + |z_2|)^3) \label{4crossclass2}
\end{gather}
where $z_1$ and $z_2$ are the amplitude of projection onto $\phi_1$ and $\phi_2$ respectively. $A$, $B$, $C$, and $D$ are continuously differentiable near $\lambda = \lambda_0$, and evaluated at $\lambda_0$, these are called the transition numbers, and we define $m_1 = -\re A/ \re B$ and $m_2 = -\re C / \re D$. The transition is then completely characterized by the dynamics of the reduced system (\ref{4crossclass1}) - (\ref{4crossclass2}), and we have the following classification:

\begin{enumerate}[label = (\roman*)]
    \item If $\re A < 0$ and $\re B < 0$, then 
    \begin{enumerate}
        \item if $\re C < 0$ and $\re D < 0$, then the transition is continuous,
        \item if $\re C < 0$ and $\re D > 0$, then the transition is mixed if $\re A - \re C > 0$ and catastrophic if $\re A - \re C \le 0$,
        
        \item if $\re C > 0$ and $\re D < 0$, then the transition is continuous,
        
        \item if $\re C > 0$ and $\re D > 0$, then the transition is catastrophic.
    \end{enumerate}
    
    \item If $\re A < 0$ and $\re B > 0$, then 
    \begin{enumerate}
        \item if $\re C < 0$ and $\re D < 0$, then the transition is continuous,
        \item if $\re C > 0$ and $\re D < 0$, then the transition is continuous if $m_1 > m_2$ and catastrophic if $m_1 < m_2$. 
        
        \item if $\re C < 0$ and $\re D > 0$, then the transition is continuous if $m_1 < m_2$. Furthermore if $m_2 > m_1$, then the transition is mixed if $\frac{\re A - \re C}{\re D - \re B} > 0$, and catastrophic if $\frac{\re A - \re C}{\re D - \re B} < 0$. 
        
        \item if $\re C > 0$ and $\re D > 0$, then the transition is catastrophic. 
    \end{enumerate}
    
    \item If $\re A > 0$ and $\re B < 0$, then
    \begin{enumerate}
        \item if $\re C < 0$ and $\re D < 0$, the transition is mixed if $\re D - \re B > 0$ and catastrophic otherwise. 
        
        \item if $\re C < 0$ and $\re D > 0$, then the transition is catastrophic if $m_1 > m_2$.
        
        \item if $\re C > 0$ and $\re D < 0$, then the transition is mixed if $\frac{\re D - \re B}{\re A - \re C} > 0$ and catastrophic otherwise. 
        
        \item if $\re C > 0$ and $\re D > 0$, then the transition is catastrophic. 
    \end{enumerate}
    
    \item if $\re A > 0$ and $\re B > 0$, then the transition is catastrophic. 
\end{enumerate}
Furthermore, if the transition is continuous, the trivial stable equilibrium point bifurcates to an $S^3$ attractor for $\lambda > \lambda_0$. 
\end{theorem}

\begin{proof}
We employ the same change of variables as \cite{double_hopf}. Let
\[z_1(t) = \sqrt{\rho_1(t)} e^{i\gamma_2(t)} \qquad z_2(t) = \sqrt{\rho_2(t)} e^{i \gamma_2(t)}\]
where $\gamma_i$ and $\rho_i$ for $i = 1, 2$ are real functions, and $\rho_i(t) \ge 0$. Then substituting into \cref{4crossclass1} and \cref{4crossclass2} then taking the real part, we find that 
\begin{gather*}
    \dot \rho_1 = 2 \rho_1(\re \beta_1 + \re(A) \rho_1 + \re(B) \rho_2 + O(\rho_1^2 + \rho_2^2)), \\
    \dot \rho_2 = 2 \rho_2(\re \beta_1 + \re(C) \rho_1 + \re(D) \rho_2 + O(\rho_1^2 + \rho_2^2)).
\end{gather*}
Note that since $\dot \rho_1|_{\rho_1 = 0} = 0$ and $\dot \rho_2|_{\rho_2 = 0} = 0$, any initial value of $(\rho_1, \rho_2)$ will remain in the first quadrant. It is clear that if all trajectories with initial value in the first quadrant tend to $\infty$ as $t \to \infty$, then the transition will be catastrophic, and if all such trajectories tend to $0$, then the transition will be continuous. Otherwise it is mixed. 

\begin{enumerate}[label = (\roman*)]
    \item First consider the case $\re A, \re B < 0$. Then $\dot \rho_1 < 0$ in the first quadrant. Thus if $\re C, \re D < 0$, we have a continuous transition, and if $\re C, \re D < 0$, we have a catastrophic transition. 
    
    Now suppose $\re C > 0$ and $\re D < 0$. Then on the left side of the line $\rho_2 = m_2 \rho_1$, $\dot \rho_2 < 0$, and on the right side $\dot \rho_2 > 0$. Clearly every trajectory that starts on the left side tends to $0$. Since $|\dot \rho_1|$ increases as $\rho_2$ increases, every trajectory that starts on the right side must enter the left side, and thus tends to $0$ as $t \to \infty$. 
    
    Finally consider $\re C < 0$ and $\re D > 0$. In this case, on left side of the line $\rho_2 = m_2 \rho_1$, we have $\dot \rho_2 > 0$ and on the right we have $\dot \rho_2 < 0$. Any solution starting on the left side of the line must then tend to $\infty$. If $\re A - \re C > 0$, then the line $\rho_2 = k \rho_1$ where $k = \frac{\re A - \re C}{\re D - \re B}$ has positive slope and thus passes through the first quadrant. In particular, $k < m_2$, along this line, $\dot \rho_2/\dot \rho_1 = k$. Therefore any trajectory with initial value below this line must tend to $0$, hence the transition is mixed. On the other hand, if $\re A - \re C \le 0$, then $\dot \rho_2/\dot \rho_1 < \rho_2/\rho_1$ for every point below the line $\rho_2 = m_2 \rho_1$. Therefore all such trajectory must enter the region above the line at some $t > 0$, thus the transition must be catestrophic.

    \item Now consider the case $\re A < 0$ and $\re B > 0$. Then on the left side of the line $\rho_2 = m_2 \rho_1$, $\dot \rho_1 > 0$, and on the right side, $\dot \rho_2 < 0$. If $\re C < 0$ and $\re D < 0$, this is identical to case 1(c), so the transition is continuous.
    
    If $\re C > 0$ and $\re D < 0$, then above the line $\rho_2 = m_2 \rho_1$, $\dot \rho_2 < 0$, and below, we have $\dot \rho_2 > 0$. In \cref{fig:m1_m2_interaction}, we clearly see that there are two cases depending on if $m_1 < m_2$ or $m_1 < m_2$. The first quadrant is divided into three regions by the two lines. Label the regions I, II, and III from left to right. In either case, initial values in regions I and III will end up in region II. In the case $m_1 < m_2$, initial values in region II will tend to $\infty$, and in the case $m_1 > m_2$, initial values in region II will tend to $0$. Therefore we have catastrophic transition in the former case and continuous in the latter. 
    
    \begin{figure}
        \centering
        \includegraphics[scale = 0.5]{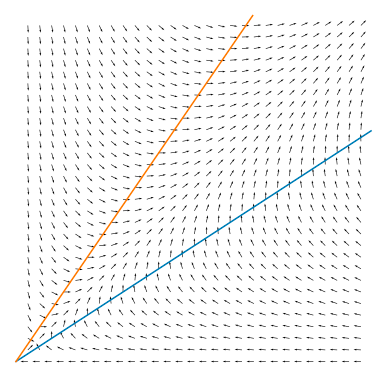}
        \includegraphics[scale = 0.5]{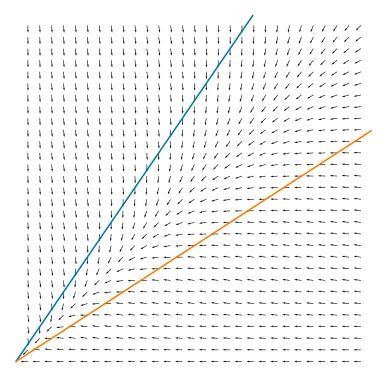}
        \caption{These vector fields show the general shape of the case $\re C > 0$ and $\re D < 0$. The blue line is the line $\rho_2 = m_1 \rho_1$ and the orange line is the line $\rho_2 = m_2 \rho_1$.}
        \label{fig:m1_m2_interaction}
    \end{figure}
    
    A combination of the analysis of 2(b) and 1(b) proves 2(c). Finally if $\re C > 0$ and $\re D > 0$, the trajectory will necessarily tend to $\infty$ along the $\rho_2$ direction, hence it is catastrophic. 
\end{enumerate}
(iii) and (iv) can be shown through identical analysis. See \cite{ma2005bifurcation} for a proof of the $S^3$ structure of the bifurcated attractor. 
\end{proof}

For the sake of completeness completeness, \cref{th:double_hopf} fully characterized every possible combination of the transition number. We will see in \cref{section:double} that our Swift-Hohenberg system with dispersion will only require analysis of a subset of these combinations.

Now to characterize the transitions of \cref{SH_system}, we must identify the eigenvalues that first cross the imaginary axis of the associated operator $L_\lambda$ defined in \cref{operator}. Expanding $L_\lambda u = \beta(\lambda) u$ as a Fourier series, we find that we have eigenvalues
\[\beta_n(\lambda) = \lambda - \left(1 - \frac{4 \pi^2 n^2}{\ell^2} \right)^2 - i \sigma \frac{8 \pi^3}{\ell^3} n^3\]
with corresponding eigenvectors 
\[\phi_n(x) = e^{2 \pi i nx/\ell}.\]
for $n \in \Z$. It is clear that $\argmax_{n \in \Z} \re \beta_n(0)$ is always a finite set with some cardinality $m(\ell)$, and it consists of the first eigenvalues to cross the imaginary axis (see \cref{fig:eig_plot}).
\begin{figure}
    \centering
    \includegraphics[scale = 0.7]{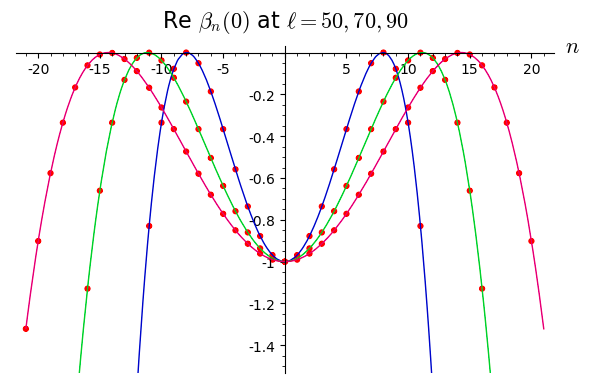}
    \caption{Plot of $\re \beta_n(0)$ for various $\ell$. One can see from the figure that for large $\ell$, the maximas over $n \in \Z$ are attained at one or two pairs of conjugate eigenvalues. For sufficiently small $\ell$, $\beta_0(0)$ will be the maximum eigenvalue.}
    \label{fig:eig_plot}
\end{figure}
Observe that $\beta_0(0) = -1$, and for all $\ell < 2 \pi/\sqrt{2}$, $\re \beta_n (0) < -1$ for all $n \neq 0$, and for $\ell > 2 \pi/2$, $\re \beta_1(0) > -1$. Therefore
\[\begin{cases} m(\ell) = 1 & \ell < 2 \pi/\sqrt{2} \\ m(\ell) \ge 2 & \ell \ge 2\pi/\sqrt{2} \end{cases}.\]
In particular, $m(2 \pi/\sqrt{2}) = 3$ since $\re \beta_1 = \re \beta_{-1} = \re \beta_0 = -1$. When $\ell > 2\pi / \sqrt{2}$, note that $\re \beta_1(0) > -1$, so for such $\ell$, we either have one or two pairs of conjugate eigenvalues that first cross the imaginary axis as $\lambda$ increases. We then define the following partition of $\R^+$:
\begin{definition}
Let $\mathcal I_1 \cup \mathcal I_2 \cup \mathcal I_3 \cup \mathcal I_4 = \R^+$ so that for any $\ell \in \mathcal I_i$, $m(\ell) = i$. In particular, $\mathcal I_1$ is an open interval, $\mathcal I_2$ is a union of open intervals, $\mathcal I_3$ is a single point, and $\mathcal I_4$ is a discrete set (see \cref{fig:partition}). 
\end{definition}

\begin{figure}
    \centering
    \includegraphics[scale = 0.7]{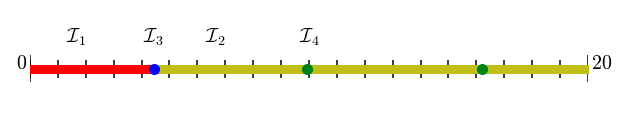}
    \caption{A visual of the partition. $\mathcal I_1$, $\mathcal I_2$, $\mathcal I_3$, and $\mathcal I_4$ are encoded by the colors red, yellow, blue, and green respectively.}
    \label{fig:partition}
\end{figure}
For convenience, we define the inverted length
\[\rho \coloneqq \frac{2 \pi}{\ell}\]
If $\beta_k(0) \in \argmax_{n \in \Z} \re \beta_n(0)$, then at $\lambda_0 = (1 - (k \rho)^2)^2$, then we have a dynamic transition. Now we are ready to state the main theorem.
\begin{theorem}\label{MAIN}
The Swift-Hohenberg equation with third-order dispersion always undergoes a dynamic transition across $\lambda = \lambda_0$, and the transitions can be characterized by the following:
\begin{enumerate}[label = (\roman*)]
    \item if $\ell \in \mathcal I_1$, then the transition is mixed if $b \neq 0$, and is continuous if $b = 0$. In particular, if $b = 0$, the basic state at $u = 0$ will bifurcate to two stable equilibrium points for $\lambda > \lambda_0 = 1$.
    
    \item if $\ell \in \mathcal I_2$, then there exists $k$ such that $\re \beta_k(\lambda_0) = 0$ where $\lambda_0 = (1 - (k \rho)^2)^2$. The transition is continuous if 
    \[\frac{2b^2(15 (k\rho)^4 - 6 (k\rho)^2)}{(15 (k\rho)^4 - 6 (k\rho)^2)^2 + (6 \sigma (k\rho)^3)^2} + \frac{4 b^2}{1 - \lambda_0} - 3 < 0,\]
    and catastrophic if the inequality is reversed. If continuous, the basic attractor at $u = 0$ bifurcates to a stable periodic orbit for $\lambda > \lambda_0$, and if catastrophic, the basic attractor at $u = 0$ bifurcates to an unstable periodic orbit for $\lambda < \lambda_0$. 
    
    \item if $\ell \in \mathcal I_3$, then the transition is mixed if $b \neq 0$, and is continuous if $b = 0$. In particular, if $b = 0$, the basic state at $u = 0$ will bifurcate to an $S^2$ attractor $\Sigma_\lambda$ for $\lambda > \lambda_0 = 1$. Furthermore, $\Sigma_\lambda$ contains two stable equilibrium points and an unstable periodic orbit. 
    
    \item if $\ell \in \mathcal I_4$, then there exists $k$ such that $\re \beta_i(\lambda_0) = 0$, $i = k, k + 1$ with $\rho^2 = 2/(k^2 + (k + 1)^2)$ and $\lambda_0 = (1 - (k \rho)^2)^2$. Define parameters
    \begin{gather}
        \eta_1 = \frac{4b^2}{1 - \lambda_0} + \frac{2b^2 (15 (k\rho)^4 - 6(k\rho)^2)}{(15 (k\rho)^4 - 6(k\rho)^2)^2 + (6 \sigma (k\rho)^3)^2} - 3 \label{eta1}\\
        \eta_2 = \frac{4b^2}{1 - \lambda_0} + \frac{4b^2((2k^2 - 2) \rho^2 - (k^4 - 1)\rho^4)}{((2k^2 - 2) \rho^2 - (k^4 - 1)\rho^4)^2 + (\sigma \rho^3(k^2 + 3k))^2} - 6 \\
        \eta_3 = \frac{4b^2}{1 - \lambda_0} + \frac{2b^2 (15 (k\rho + \rho)^4 - 6(k\rho + \rho)^2)}{(15 (k\rho + \rho)^4 - 6(k\rho + \rho)^2)^2 + (6 \sigma (k\rho + \rho)^3)^2} - 3. \label{eta2}
    \end{gather}
    If $k \ge 2$, then the transitions are as follows:
    
\begin{table}[h]
\begin{tabular}{|l|l|l|}
\hline
                           & $\eta_2 > 0$                                                                                               & $\eta_2 < 0$                                                                                               \\ \hline
$\eta_1 > 0$, $\eta_3 > 0$ & catastrophic                                                                                               & \begin{tabular}[c]{@{}l@{}}continuous if $\eta_1 \eta_3 < \eta_2^2$,\\ catastrophic otherwise\end{tabular} \\ \hline
$\eta_1 > 0$, $\eta_3 < 0$ & catastrophic                                                                                               & \begin{tabular}[c]{@{}l@{}}mixed if $\eta_3 > \eta_2$\\ catastrophic otherwise\end{tabular}                \\ \hline
$\eta_1 < 0$, $\eta_3 < 0$ & \begin{tabular}[c]{@{}l@{}}continuous if $\eta_1 \eta_3 > \eta_2^2$,\\ catastrophic otherwise\end{tabular} & continuous                                                                                                 \\ \hline
\end{tabular}
\end{table}
    
    If $k = 1$, then the transition is continuous if $\eta_3 < 0$ and catastrophic if $\eta_3 > 0$. If the transition is continuous, then $u = 0$ bifurcates to an $S^3$ attractor for $\lambda > \lambda_0$. 
    
\end{enumerate}
\end{theorem}

\begin{remark}
Due to the double Hopf structure, in the case that the transition for $\ell \in \mathcal{I}_4$ is continuous, the bifurcated attractor $\Sigma_\lambda$ contains two distinct periodic orbits as well as an invariant torus. The stability of these on the attractor depends on the transition parameters. See \cite{ma2005bifurcation} for more details. 
\end{remark}

\section{Dynamic Transitions}\label{section:dynamic_transitions}
Now we are ready to give a full characterization of the dynamic transitions of \cref{SH_system}. 
\subsection{Real Simple Transition}\label{section:real_simple}
In the case that $\ell \in S_1$ or $S_2$, we have real and complex simple eigenvalue transitions respectively. First consider the case that $m(\ell) = 1$.
\begin{proposition}
If $\ell < 2 \pi/\sqrt{2}$, then the transition at $\lambda_0$ is mixed if $b \neq 0$, and is continuous if $b = 0$. 
\end{proposition}
\begin{proof}
The center manifold is trivial. The trajectories on the center manifold simply consists of the solutions that depend only on time. At the transition $\lambda_0 = 1$, projecting onto the center manifold then gives the reduced system
\[\frac{du}{dt} = b u^2 - u^3\]
which implies that if $b > 0$, solutions with initial value $u(0) < 0$ will tend toward the fixed point at $u = 0$, and solutions with initial vaue $u(0) > 0$ will be repelled away. If $b < 0$ then we have the opposite behavior. Thus the transition is mixed for $b \neq 0$. If $b = 0$, $u = 0$ is then an asymptotically stable fixed point. Thus the transition will be continuous if $b = 0$.  
\end{proof}

\subsection{Complex Simple Transition} \label{section:complex_simple}

Now consider the transition from complex simple eigenvalues. There will be a pair of conjugate eigenvalues that cross the imaginary axis as $\lambda$ crosses $\lambda_0$. Then there exists $\lambda_0 > 0$ and $k > 0$ so that $\beta_k = \overline{\beta}_{-k}$ and  
\[\begin{array}{ll}
    \re(\beta_i) = \begin{cases} < 0 & \text{if } \lambda < \lambda_0 \\ = 0 & \text{if } \lambda = \lambda_0 \\ > 0 & \text{if } \lambda > \lambda_0 \end{cases} & \qquad i = \pm k\\
    \re(\beta_i) < 0 & \qquad i \neq \pm k
\end{array}\]
Similar to before, we can split into subspaces
\[\begin{array}{c} E_1 = \text{span}\{z \phi_k + \overline{z \phi_1} : z \in \C\} \\ E_2 = E_1^\perp \quad \text{in } H^4(\T) \\ \overline{E}_2 \quad \text{is the closure in } L^2(\T) \end{array},\]
and split the operator into $L_\lambda = \mathcal{J}_\lambda \oplus \mathcal{L}_\lambda$ so that $\mathcal{J}_\lambda = \text{Diag}(\beta_k, \overline{\beta}_k)$. 
\begin{proposition}
The transition number for all $\ell \in S_2$ is given by 
\begin{equation}
    P = \frac{2b^2}{(15 ^4 - 6 (k\rho)^2) + 6 \sigma (k\rho)^3 i} + \frac{4b^2}{1 - \lambda_0} - 3.
\end{equation}
Thus if $b$ and $\sigma$ are such that 
\begin{equation}
    \frac{2b^2(15 (k\rho)^4 - 6 (k\rho)^2)}{(15 (k\rho)^4 - 6 (k\rho)^2)^2 + (6 \sigma (k\rho)^3)^2} + \frac{4 b^2}{1 - \lambda_0} - 3 < 0
\end{equation}
then the transition at $\lambda = \lambda_0$ will be continuous, for the reverse inequality, the transition will be catastrophic. 
\end{proposition}
\begin{proof}
We begin by computing the center manifold. The center manifold is locally given by the graph of some map $\Phi(\cdot, \lambda): E_1 \to \overline{E}_2$, which can be approximated by
\begin{equation}\label{center_formula}
    \Phi(u, \lambda) = \int_{-\infty}^0 e^{-\tau \mathcal{L}_\lambda} (k\rho)_\epsilon P_2 G_2 (e^{\tau \mathcal{J}_\lambda} u ) \, d\tau + o(\lVert u \rVert^2)
\end{equation}
where $P_2$ is the projection on to $\overline{E}_2$, $u = z \phi_k + \overline{z \phi_k}$ for some $z \in \C$, and $\rho_\epsilon$ is a cutoff function so that we have compact support in $E_1$. This formula is a generalization of the standard Lyapunov-Perron construction of invariant manifolds, see \cite{appendix, superconductivity,ma2005bifurcation} for more details. First observe that
\[P_2 G(e^{\tau \mathcal{J}_{\lambda}} u) = b e^{2 \tau \beta_k} z^2 \phi_k^2 + b e^{2 \tau \overline{\beta}_k} \overline{z^2 \phi_k^2} + 2b e^{\tau(\beta_k + \ol \beta_k)} z \phi_k \ol{z \phi_k}.\]
For $\lambda$ near $\lambda_0$, $\text{Re} (2\beta_k - \beta_i) > 0$ and $\text{Re}(\beta_k + \overline{\beta}_k - \beta_i) > 0$ for all $i \neq \pm k$. Therefore the integral in \cref{center_formula} converges. Then the center manifold function is given by 
\begin{align}
    \Phi(u, \lambda) =& \frac{b}{2 \beta_k - \beta_{2k}} z^2 \phi_k^2 + \frac{b}{\ol{2 \beta_k - \beta_{2k}}} \ol{z^2 \phi_k}^2 + \frac{2b}{\beta_k + \ol \beta_k - \beta_0} z \phi_k \ol{z\phi_k} + o ( \| u \|^2) \nonumber \\
    =& \frac{b}{[\lambda - (1 - (k\rho)^2)^2] - (6 (k\rho)^2 - 15(k\rho)^4) + 6 \sigma (k\rho)^3 i} z^2 \phi_k^2 \nonumber \\
    &+ \frac{b}{2[\lambda - (1 - (k\rho)^2)^2] + 1 - \lambda} |z|^2 + o(\|u\|^2) + c.c. \nonumber \\
    \eqqcolon & A z^2 \phi_k^2 + B |z|^2 + o(\lVert u \rVert^2) + c.c. \label{complex_center_func_2}
\end{align}
where $c.c.$ is the complex conjugate of everything before it. Let $\PP_k$ denote the projection onto $\phi_k$, then 
\begin{equation*}
    \PP_k G(z\phi_k + \ol z \ol \phi_k + \Phi(u, \lambda)) = 2b A z |z|^2 + 4b Bz |z|^2 - 3z|z|^2 + o(\lVert u \Vert^2)
\end{equation*}
Thus the dynamics on the center manifold projected onto $\phi_k$ at $\lambda = \lambda_0$ is given by the equation
\begin{align*}
    \frac{dz}{dt} =& \PP_k L_{\lambda_0}(z\phi_k + \ol z \ol \phi_k + \Phi(u, \lambda))+ \PP_k G(z\phi_k + \ol z \ol \phi_k + \Phi(u, \lambda)) \\
    =& \beta_k(\lambda_0) z + (2b A + 4b B - 3) z|z|^2 + o(\lVert u \rVert^3).
\end{align*}
This indeed satisfies all the assumptions of \cref{th:single_hopf}, and thus the transition number is given by $P = 2bA + 4b B - 3$. Note that 
\[\lambda_0 - (1 - (k\rho)^2)^2 = 0\]
is it is when the eigenvalues cross the imaginary axis. Therefore we obtain the proposition upon substitution. 
\end{proof}

Part (iii) of \cref{th:single_hopf} gives the bifurcated periodic solutions. So for continuous transitions, the system bifurcates to a stable periodic orbit for $\lambda > \lambda_0$, and for catastrophic transitions, the system bifurcates to an unstable periodic orbit for $\lambda < \lambda_0$.

\subsection{Double Real-Complex Transition}
Now we consider the single case of $\ell \in S_3$. Let $E_1 = \text{span}\{z_0 + z_1 \phi_1 + \ol{z_1 \phi_1} : z_0 \in \R, z_1 \in \C\}$, and define $E_2$ as before. Split $L_\lambda$ into $\mathcal J_\lambda$ and $\mathcal L_\lambda$ as before. 

\begin{proposition}
If $\ell \in S_3$, then the transition across $\lambda = \lambda_0$ is mixed for $b \neq 0$, and is continuous for $b = 0$.  
\end{proposition}

\begin{proof}
We express $u \in E_1$ as $u = z_0 + z_1 \phi_1 + \ol{z_1 \phi_1}$ where $z_0 \in \R$ and $z_1 \in \C$. Then 
\begin{align*}
    P_2 G_2(e^{\tau \mathcal J_\lambda} u) =& P_2\left[ b(e^{\tau \beta_0} z_0 + e^{\tau \beta_1} z_1 \phi_1 + e^{\tau \ol \beta_1} \ol z_1 \ol \phi_1)^2 \right] \\
    =& be^{2 \tau \beta_1} z_1^2 \phi_1^2 + be^{2 \tau \ol \beta_1} \ol z_1^2 \ol \phi_1^2
\end{align*}
Then the center manifold function is given by 
\begin{equation*}
    \Phi(u, \lambda) = \frac{b}{2 \beta_1 - \beta_2} z_1^2 \phi_1^2 + \frac{b}{2 \ol \beta_1 - \ol \beta_2} \ol z_1^2 \ol \phi_1^2 + o((|z_0| + |z_1|)^2)
\end{equation*}
Note that if we project the dynamics onto $\phi_0$, we find
\[\frac{dz_0}{dt} = \beta_0 z_0 + \PP_0(G(u + \Phi(u, \lambda))) = \beta_0 z_0 + b z_0^2 - z_0^3.\]
At $\lambda = \lambda_0 = 1$, this is the same situation as \cref{section:real_simple}. Therefore for $b > 0$, trajectories with initial value whose $\phi_0$ projection is positive will tend away from $0$, and trajectories with initial value whose $\phi_0$ projection is negative will tend toward $0$. For $b < 0$, we again have the opposite behavior. Thus the transition is mixed if $b \neq 0$.

If $b = 0$, then $\Phi(u, \lambda) = o((|z_0| + |z_1|)^2)$. Then projecting onto $\phi_0$ and $\phi_1$, we have the reduced system
\begin{align}
    & \frac{dz_0}{dt} = \beta_0 z_0 - z_0^3 \label{3proj_b=0_z0} \\ 
    & \frac{dz_1}{dt} = \beta_1 z_1 - z_0^2 z_1 - z_1 |z_1|^2 + o((|z_0| + |z_1|)^3). \label{3proj_b=0_z1}
\end{align}
Again we perform the change of variables to $z_1(t) = r(t) e^{i \gamma(t)}$ where $r(t) > 0$ and $\gamma$ is a real function. Then substituting into \cref{3proj_b=0_z1} and taking the real part at $\lambda = \lambda_0 = 1$, 
\[\frac{dr}{dt} = -z_0^2 r - r^3 + o((|z_0| + r)^3).\]
Since $r > 0$, this implies $\dot r < 0$. \cref{3proj_b=0_z0} implies that $\dot z_0 < 0$ as well, so $u = 0$ at $\lambda = \lambda_0$ is an asymptotically stable fixed point. Therefore the transition is continuous. 
\end{proof}

\subsection{Double Complex Transition}\label{section:double}
Fix $\ell$ so that for some $k \ge 1$, 
\[\begin{array}{ll}
    \re(\beta_i) = \begin{cases} < 0 & \text{if } \lambda < \lambda_0 \\ = 0 & \text{if } \lambda = \lambda_0 \\ > 0 & \text{if } \lambda > \lambda_0 \end{cases} & \qquad i = \pm k, \pm(k + 1)\\
    \re(\beta_i) < 0 & \qquad i \neq \pm k, \pm (k + 1)
\end{array}\]
Define subspaces
\begin{gather*}
    E_1 = \text{span} \{z_1 \phi_k + \ol{z_1 \phi_1}, z_2 \phi_{k + 1} + \ol{z_2 \phi_{k + 1}}: z_1, z_2 \in \C\} \\
    E_2 = E_1^\perp \text{ in } H^4(\T) \\
    \ol{E}_2 \quad \text{is the closure in }L^2(\T).
\end{gather*}
Split $L_\lambda = \mathcal{J}_\lambda \oplus \mathcal{L}_\lambda$ so that $\mathcal{J}_\lambda = \text{Diag}(\beta_k, \ol \beta_k, \beta_{k + 1}, \ol \beta_{k + 1})$. Then for $u \in E_1$ (for which we can write as $z_1 \phi_k + \ol{z_1 \phi_1} + z_2 \phi_{k + 1} + \ol{z_2 \phi_{k + 1}}$ for some $z_1, z_2 \in \C$. 
\begin{proposition}
The system near $\lambda = \lambda_0$ is completely characterized by the projected system 
\begin{gather}
    \frac{dz_1}{dt} = \beta_1(\lambda) z_1 + z_1(A(\lambda)|z_1|^2 + B(\lambda)|z_2|^2) + o((|z_1| + |z_2|)^3) \\
    \frac{dz_2}{dt} = \beta_2(\lambda) z_2 + z_2(C(\lambda)|z_1|^2 + D(\lambda)|z_2|^2) + o((|z_1| + |z_2|)^3)
\end{gather}
If $k > 1$, then the transition numbers are given by 
\begin{gather*}
    A = \frac{4b^2}{1 - \lambda_0} + \frac{2b^2}{(15 (k\rho)^4 - 6(k\rho)^2) + 6 \sigma (k\rho)^3 i} - 3 \\
    B = \frac{4b^2}{1 - \lambda_0} + \frac{4b^2}{(2k^2 - 2) \rho^2 - (k^4 - 1)\rho^4 + \sigma (3k^2 + 3k)\rho^3 i} - 6 \\
    C = \frac{4b^2}{1 - \lambda_0} + \frac{4b^2}{(2k^2 - 2) \rho^2 - (k^4 - 1)\rho^4 - \sigma (3k^2 + 3k)\rho^3 i} - 6 \\
    D = \frac{4b^2}{1 - \lambda_0} + \frac{2b^2}{(15 (k\rho + \rho)^4 - 6(k\rho + \rho)^2) + 6 \sigma (k\rho + \rho)^3 i} - 3
\end{gather*}
and if $k = 1$, the transition numbers are given by 
\begin{gather*}
    A = \frac{4b^2}{1 - \lambda_0} - 3 \\
    B = C = \frac{4b^2}{1 - \lambda_0} - 6 \\
    D = \frac{4b^2}{1 - \lambda_0} + \frac{2b^2}{(15 (2\rho)^4 - 6(2\rho)^2) + 6 \sigma (2\rho)^3 i} - 3
\end{gather*}

\end{proposition}

\begin{proof}
Again we first compute the center manifold function. Observe that for $u \in E_1$, 
\begin{multline*}
    G_2(e^{\tau \mathcal{J}_\lambda} u) = e^{2 \tau \beta_k} z_1^2 \phi_k^2 + e^{2 \tau \ol \beta_k} \ol z_1^2 \ol \phi_k^2 + e^{2 \tau \beta_{k + 1}} z_2^2 \phi_{k + 1}^2 + e^{2 \tau \ol \beta_{k + 1}} \ol z_2^2 \ol \phi_{k + 1}^2 \\ + 2 \left(e^{\tau (\beta_k + \ol \beta_k)} z_1 \ol z_1 \phi_k \ol \phi_k + e^{\tau \beta_{k + 1} + \ol \beta_{k + 1}} z_2 \ol z_2 \phi_{k + 1} \ol \phi_{k + 1} \right. \\ \left. + e^{\tau(\beta_k + \ol \beta_{k + 1})} z_1 \ol z_2 \phi_k  \ol \phi_{k + 1} + e^{\tau(\ol \beta_k + \beta_{k + 1})} \ol z_1 z_2 \ol \phi_k \phi_{k + 1} \right)
\end{multline*}
It is clear from above that $P_2 G_2(e^{\tau \mathcal{J}_\lambda} u) = G_2(e^{\tau \mathcal{J}_\lambda} u)$ for $k \neq 1$, and for $k = 1$, several cross terms will vanish. We consider the case $k \neq 1$ first. Projecting onto $E_2$ and integrating $\tau$ from $-\infty$ to $0$, we find 
\begin{multline}\label{4cross_center_func}
    \Phi(u, \lambda) = \frac{b}{2 \beta_k - \beta_{2k}} z_1^2 \phi_k^2 + \frac{b}{2 \beta_{k + 1} - \beta_{2k + 2}} z_2^2 \phi_{k + 1}^2 + \frac{b}{\beta_k + \ol \beta_k - \beta_0} |z_1|^2 \\ 
    + \frac{b}{\beta_{k + 1} + \ol \beta_{k + 1} - \beta_0} |z_2|^2 + \frac{2b}{\ol \beta_k + \beta_{k + 1} - \beta_1} \ol z_1 z_2 \phi_1 + o((|z_1| + |z_2|)^2) + c.c. 
\end{multline}
where $c.c.$ is the complex conjugate of everything preceding it. Label the coefficients in \cref{4cross_center_func} $\mu_i$ for $i$ from $1$ to $5$ so that 
\[\Phi(u, \lambda) = \mu_1 z_1^2 \phi_k^2 + \mu_2 z_2^2 \phi_{k + 1}^2 + \mu_3|z_1|^2 + \mu_4 |z_2|^2 + \mu_5 \ol z_1 z_2 \phi_1 + c.c.\]
Note that $\mu_3$ and $\mu_4$ are real, and all of them depend on $\lambda$ and the fixed parameters $b$ and $\sigma$. Note that at $\lambda_0 - (1 - (k \rho)^2)^2 = 0$ and $\lambda_0 - (1 - (k \rho + \rho)^2)^2 = 0$. So at $\lambda = \lambda_0$, the coefficients are given by
\begin{gather*}
    \mu_1 = \frac{b}{15(k \rho)^4 - 6 (k \rho)^2 + 6\sigma (k \rho)^3 i} \\ 
    \mu_4 = \frac{b}{15(k \rho + \rho)^4 - 6 (k \rho + \rho)^2 + 6\sigma (k \rho + \rho)^3 i} \\
    \mu_2 = \mu_3 = \frac{b}{1 - \lambda_0} \\
    \mu_5 = \frac{2b}{(2k^2 - 2) \rho^2 - (k^4 - 1)\rho^4 - \sigma (3k^2 + 3k)\rho^3 i}
\end{gather*}
The complex system yielded from projecting the dynamics on the center manifold onto $\phi_k$ and $\phi_{k + 1}$ completely captures the dynamics of the real system. Let $\PP_k$ and $\PP_{k + 1}$ denote the projections onto $\phi_k$ and $\phi_{k + 1}$ respectively. Then 
\begin{multline*}
    \PP_k G (z_1 \phi_k + z_2 \phi_{k + 1} + \ol{z_1 \phi_k} + \ol{z_2 \phi_{k + 1}} + \Phi(u, \lambda)) = (4b\mu_3 + 2 b \mu_1 - 3) z_1|z_1|^2 \\ + (4 b \mu_4 + 2 b \ol \mu_5 - 6) z_1 |z_2|^2 + o((|z_1| + |z_2|)^3)
\end{multline*}
and 
\begin{multline*}
    \PP_{k + 1} G (z_1 \phi_k + z_2 \phi_{k + 1} + \ol{z_1 \phi_k} + \ol{z_2 \phi_{k + 1}} + \Phi(u, \lambda)) = (4b\mu_3 + 2b\mu_5 - 6) z_2 |z_1|^2 \\ + (4b \mu_4 + 2b \mu_2 - 3) z_2 |z_2|^2 + o((|z_1| + |z_2|)^3)
\end{multline*}

Now for $k = 1$, 
\begin{equation*}
    P_2 G_2(e^{\tau \mathcal J_\lambda} u) = e^{2 \tau \beta_2} z_2 \phi_2^2 + e^{\tau (\beta_1 + \ol \beta_1)} |z_1| + e^{\tau (\beta_2 + \ol \beta_2)} |z_2| + c.c.
\end{equation*}
So then 
\begin{equation*}
    \Phi(u, \lambda) = \frac{b}{2 \beta_{2} - \beta_{4}} z_2^2 \phi_2^2 + \frac{b}{\beta_1 + \ol \beta_1 - \beta_0} |z_1|^2 + \frac{b}{\beta_2 + \ol \beta_2 - \beta_0} |z_2|^2 + o((|z_1| + |z_2|)^2) + c.c. 
\end{equation*}
Label the coefficients $\nu_i$ so that $\Phi(u, \lambda) = \nu_1 z_2^2 \phi_2^2 + \nu_2 |z_1|^2 + \nu_3 |z_2|^2 + o((|z_1| + |z_2|)^2) + c.c.$. Then 
\begin{multline*}
    \PP_1 G(z_1 \phi_1 + z_2 \phi_2 + \ol{z_1 \phi_1} + \ol{z_2 \phi_2} + \Phi(u, \lambda)) = (4b \nu_2 - 3)z_1 |z_1|^2 \\ + (4b \nu_3 - 6) z_1 |z_2|^2
    + o((|z_1| + |z_2|)^3)
\end{multline*}
\begin{multline*}
    \PP_2 G(z_1 \phi_1 + z_2 \phi_2 + \ol{z_1 \phi_1} + \ol{z_2 \phi_2} + \Phi(u, \lambda)) = (4b \nu_2 - 6) z_2 |z_1|^2 \\ + (4b \nu_3 + 2b \nu_1 - 3) z_2 |z_2|^2 + o((|z_1| + |z_2|)^3)
\end{multline*}
Substituting in the values for $\mu_i$ and $\nu_j$ at $\lambda = \lambda_0$, we have the desired result.  
\end{proof}

\begin{figure}
    \centering
    \includegraphics[scale = 0.6]{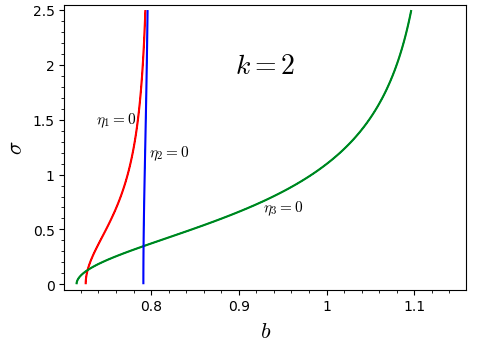}
    \caption{Phase diagram for $k = 2$.} 
    \label{fig:k=2}
\end{figure}
\begin{figure}
    \centering
    \includegraphics[scale = 0.6]{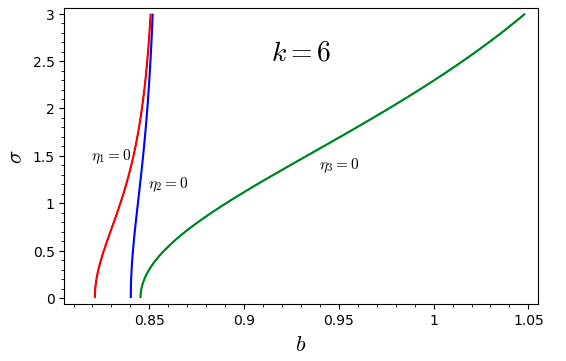}
    \caption{Phase diagram for $k = 6$.} 
    \label{fig:k=6}
\end{figure}

Note that $\re B = \re C$. The parameters given in \cref{eta1} - (\ref{eta2}) are given by
\[\eta_1 = \re A, \qquad \eta_2 = \re B = \re C, \qquad \eta_3 = \re D.\]
Note that for a fixed value of $\sigma$, $\eta_1(b) > \eta_3(b)$. Then applying \cref{th:double_hopf} immediately yields part (iv) of \cref{MAIN}.

\section{Real Center Manifold Function and Numerical Results}
So far, to classify the types of transition, it has benefitted us to work in a complex space since rotational behavior is much more simply described in complex spaces, thus giving us simple forms for the transition number (compared to the standard bifurcation number formula \cite{ma2005bifurcation} for Hopf bifucation). However, the physics of the system still resides in a real space. A particularly simple case with nontrivial transitions to analyze is the case $\ell = 2 \pi$. In this section, we give the explicit formula for the center manifold function in a real function space for this case, and numerically compute the approximate trajectories on the center manifold. 

Define $E_1$ and $E_2$ as in \cref{section:complex_simple}.
\begin{proposition}
The center manifold function $\Phi : E_1 \to \overline{E}_2$ of \cref{SH_system} with $\ell = 2 \pi$  is locally approximated by 
\begin{equation}\label{center_manifold_function}
    \Phi(u, \lambda)(x) = A(u) \cos(2x) + B(u) \sin(2x) + C(u) + o(\lVert u \rVert^2)
\end{equation}
where $u = u_1 \cos(x) + u_2 \sin(x)$ and
\begin{align*}
    & A(u) = A(u_1, u_2) \coloneqq \frac{1}{2} \frac{b}{(\lambda + 9)^2 + (6 \sigma)^2} \left((9 + \lambda) (u_1^2 - u_2^2) - 12 u_1 u_2 \sigma \right) \\
    & B(u) = B(u_1, u_2) \coloneqq \frac{1}{2} \frac{b}{(\lambda + 9)^2 + (6 \sigma)^2} \left( (9 + \lambda)2 u_1 u_2 + 6 \sigma (u_1^2 - u_2^2) \right) \\
    & C(u) = C(u_1, u_2) \coloneqq \frac{b}{2 (\lambda + 1)} (u_1^2 + u_2^2)
\end{align*}
for $\lambda$ near $0$. 
\end{proposition}
Note that $A$, $B$, and $C$ are all $O(\lVert u \rVert^2)$.

\begin{proof}
Assume $\ell = 2\pi$. Then we have 
\[\beta_n(\lambda) = \lambda - (1 - n^2)^2 - i\sigma n^3.\]
From \cref{complex_center_func_2}
\begin{align}
    \Phi(u, \lambda) &= \frac{b}{2\beta_1 - \beta_2} z^2 \phi_1^2 + \frac{b}{\overline{2\beta_1 - \beta_2}} \overline{z^2 \phi_1^2} + \frac{2b}{\beta_1 + \overline{\beta}_1 - \beta_0} z\phi_1 \overline{z \phi_1} + o(\lVert u \rVert^2) \nonumber \\
    &= 2 \text{Re} \left(\frac{b}{\lambda + 9 + 6 \sigma i} z^2\phi_1^2 \right) + \frac{2b}{\lambda + 1} |z|^2 + o(\lVert u \rVert^2) \label{complex_center}
\end{align}
We wish to express the center manifold function in the o.n.b. for $E_1$ given by 
\[e_1 = \cos x \qquad e_2 = \sin x,\]
so $u = u_1 e_1 + u_2 e_2$ for some $u_1, u_2 \in \R$, and it is clear that 
\[z = \frac{1}{2}(u_1 - iu_2).\]
We want to express the center manifold function as a function of $u_1$ and $u_2$. First note that the second term of \cref{complex_center} rearranges pretty easily to 
\begin{equation}\label{real_sys_first_term}
    \frac{b}{\lambda + 1} |z|^2 = \frac{b}{4 (\lambda + 1)} (u_1^2 + u_2^2).
\end{equation}
For the first term of \cref{complex_center}, we have
\begin{align}
    \text{Re} \left[\frac{bz^2}{\lambda + 9 + 6\sigma i} \phi_1^2 \right] &= \text{Re} \left[ \frac{bz^2}{\lambda + 9 + 6\sigma i} (e_1 + i e_2)^2 \right] \nonumber  \\
    &= \frac{1}{4} \frac{b}{(\lambda + 9)^2 + (6 \sigma)^2} \text{Re} \left[(\lambda + 9 - 6 \sigma i) (u_1 - i u_2)^2 (e_1 + i e_2)^2 \right] \nonumber \\
    &= \frac{1}{4} \frac{b}{(\lambda + 9)^2 + (6 \sigma)^2} \left[ \left((9 + \lambda) (u_1^2 - u_2^2) - 12 u_1 u_2 \sigma \right) \cos(2x)\right. \nonumber \\
    &\quad \left. + \left( (9 + \lambda)2 u_1 u_2 + 6 \sigma (u_1^2 - u_2^2) \right) \sin(2x)\right] \label{real_sys_second_term}
\end{align}
The center manifold function is then simply twice the sum of the real parts of either term, so putting \cref{real_sys_first_term} and \cref{real_sys_second_term} together, 
\begin{align}
    \Phi(u_1, u_2, \lambda)(x) &= \frac{b}{2 (\lambda + 1)} (u_1^2 + u_2^2) \nonumber \\
    &\quad + \frac{1}{2} \frac{b}{(\lambda + 9)^2 + (6 \sigma)^2} \left[ \left((9 + \lambda) (u_1^2 - u_2^2) - 12 u_1 u_2 \sigma \right) \cos(2x)\right. \nonumber \\
    &\quad \left. + \left( (9 + \lambda)2 u_1 u_2 + 6 \sigma (u_1^2 - u_2^2) \right) \sin(2x)\right] + o(\lVert u \rVert^2), \nonumber
\end{align}
which completes the proof. 
\end{proof}

From the center manifold function, we now compute the reduced system. Let $\PP_1$ and $\PP_2$ denote the projection onto $e_1$ and $e_2$ respectively. 

\begin{proposition}
The dynamics on the center manifold can be locally approximated by is then given by
\begin{align}
\cfrac{du_1}{dt} =& \lambda u_1 - \sigma u_2 + 2bu_1 C(u_1, u_2) + b u_1 A(u_1, u_2) \nonumber \\
& + b u_2 B(u_1, u_2) - \frac{3}{4} u_1^3 - \frac{3}{4} u_1 u_2^2\\
\cfrac{du_2}{dt} =& \sigma u_1 + \lambda u_2 + b u_1 B(u_1, u_2) - b u_2 A(u_1, u_2) \nonumber \\
&+ 2 b u_2 C(u_1, u_2) - \frac{3}{4} u_2^3 - \frac{3}{4} u_1^2 u_2
\end{align}
where $u_1$ and $u_2$ are the amplitudes of projection onto $e_1$ and $e_2$ respectively. 
\end{proposition}

\begin{proof}
Projecting the dynamics on the center manifold onto the center subspace, we find the reduced system
\begin{equation}\label{reduced_system_formula}
    \begin{cases} \cfrac{du_1}{dt} = \PP_1 \mathcal{L}_\lambda (u_1 e_1 + u_2 e_2) + \PP_1 G(u_1 e_1 + u_2 e_2 + \Phi(u_1 e_1 + u_2 e_2, \lambda)) \\ \cfrac{du_1}{dt} = \PP_2 \mathcal{L}_\lambda (u_1 e_1 + u_2 e_2) + \PP_2 G(u_1 e_1 + u_2 e_2 + \Phi(u_1 e_1 + u_2 e_2, \lambda)) \end{cases}.
\end{equation}
Observe that 
\begin{align}
    \mathcal{L}_{\lambda}(u_1 e_1 + u_2 e_2) &= \mathcal{L}_\lambda (z \phi_1 + \overline{z \phi_1}) \nonumber \\
    &= (\lambda - i\sigma) z \phi_1 + \overline{(\lambda - i \sigma) z \phi_1} \nonumber \\
    &= (\lambda u_1 - \sigma u_2) e_1 + (\sigma u_1 + \lambda u_2) e_2 \label{eigen_proj}
\end{align}
For the projection of $G$ onto $e_1(x)$, we find that
\begin{align}
    \PP_1 G(u_1 e_1 + u_2 e_2 + \Phi) =& \PP_1(2b u_1 u_2 e_1 e_2 + 2b u_1 e_1 \Phi + 2b u_2e_2 \Phi \nonumber \\
    & + b u_1^2 e_1^2 + b u_2^2 e_2^2 - u_1^3 e_1^3 - u_2^3 e_2^3 \nonumber \\
    & - 3u_1^2 u_2 e_1^2 e_2 - 3 u_1 u_2^2 e_1 e_2^2) + o(\lVert u \rVert^3) \nonumber \\
    =& \frac{1}{\pi} \int_0^{2\pi} (2b u_1 e_1 \Phi + 2b u_2 e_2 \Phi - u_1^3 e_1^3 - 3 u_1 u_2^2 e_1 e_2^2) e_1 \, dx \nonumber  \\
    & + o(\lVert u \rVert^3) \nonumber \\
    =& 2bu_1 C(u_1, u_2) + b u_1 A(u_1, u_2) + b u_2 B(u_1, u_2) \nonumber \\
    & - \frac{3}{4} u_1^3 - \frac{3}{4} u_1 u_2^2 + o(\lVert u \rVert^3). \label{e_1_proj}
\end{align}
For the projection onto $e_2(x)$, we similarly find
\begin{align}
    \PP_2 G(u_1 e_1 + u_2 e_2 + \Phi) =& \PP_2(2b u_1 u_2 e_1 e_2 + 2b u_1 e_1 \Phi + 2b u_2e_2 \Phi \nonumber \\
    & + b u_1^2 e_1^2 + b u_2^2 e_2^2 - u_1^3 e_1^3 - u_2^3 e_2^3 \nonumber \\
    & - 3u_1^2 u_2 e_1^2 e_2 - 3 u_1 u_2^2 e_1 e_2^2) + o(\lVert u \rVert^3) \nonumber \\
    =& \frac{1}{\pi} \int_0^{2\pi} (2 b u_1 e_1 \Phi + 2 b u_2 e_2 \Phi - u_2^3 e_2^3 - 3 u_1^2 u_2 e_1^2 e_2) e_2 \, dx \nonumber \\
    & + o(\lVert u \rVert^3) \nonumber \\
    =& b u_1 B(u_1, u_2) - b u_2 A(u_1, u_2) + 2 b u_2 C(u_1, u_2) \nonumber \\
    & - \frac{3}{4} u_2^3 - \frac{3}{4} u_1^2 u_2 + o(\lVert u \rVert^3). \label{e_2_proj}
\end{align}
So then putting \cref{eigen_proj}-(\ref{e_2_proj}) into \cref{reduced_system_formula}, we find an approximate reduced system that locally describes the dynamics:
\begin{align}
\cfrac{du_1}{dt} =& \lambda u_1 - \sigma u_2 + 2bu_1 C(u_1, u_2) + b u_1 A(u_1, u_2) \nonumber \\
& + b u_2 B(u_1, u_2) - \frac{3}{4} u_1^3 - \frac{3}{4} u_1 u_2^2 \label{reduced_system_u1}\\
\cfrac{du_2}{dt} =& \sigma u_1 + \lambda u_2 + b u_1 B(u_1, u_2) - b u_2 A(u_1, u_2) \nonumber \\
&+ 2 b u_2 C(u_1, u_2) - \frac{3}{4} u_2^3 - \frac{3}{4} u_1^2 u_2 \label{reduced_system_u2}
\end{align}
\end{proof}

For completeness and to verify that this is indeed the same as what we computed in \cref{section:complex_simple}, we directly compute the Lyapunov number associated with the Hopf bifucation that occurs at $\lambda_0$. Note that \cref{e_1_proj} and \cref{e_2_proj} can be rearranged to a polynomial $u_1$ and $u_2$ so that
\[\PP_i G(u_1 e_1 + u_2 e_2 + \Phi(u_1 e_1 + u_2 e_2, \lambda_0)) = \sum_{2 \le p + q \le 3} a^{i}_{pq} u_1^p u_2^q + o(\lVert u \rVert^3)\]
for $i = 1, 2$. Then the Lyapunov number is given by 
\[\begin{array}{c}
    \eta = \cfrac{3\pi}{4} (a^1_{30} + a^2_{03}) + \cfrac{\pi}{4} (a^1_{12} + a^2_{21}) + \cfrac{\pi}{2 \text{Im}(\overline{\beta}_1)} (a^1_{02} a^2_{02} - a^1_{20} a^2_{20}) \\
    + \cfrac{\pi}{4 \text{Im}(\overline{\beta}_1)} (a^1_{11} a^1_{20} + a^1_{11} a^1_{02} - a^2_{11} a^2_{20} - a^2_{11} a^2_{02} ).
\end{array}\]
It is easy to see by inspection that $a^i_{pq} = 0$ for all $p + q = 2$, and 
\begin{gather*}
    a^1_{30} = a^2_{03} = -\frac{3}{4} + b^2 + \frac{1}{2} \frac{b^2}{9 + 4 \sigma^2} \\
    a^1_{12} = a^2_{21} = -\frac{3}{4} + b^2 + \frac{1}{2} \frac{b^2}{9 + 4 \sigma^2}
\end{gather*}
Then following the bifurcation number formula, we find
\[\eta = \pi \left(2b^2 + \frac{b^2}{9 + 4 \sigma^2} - \frac{3}{2} \right)\]
\begin{figure}
    \centering
    \includegraphics[scale = 0.6]{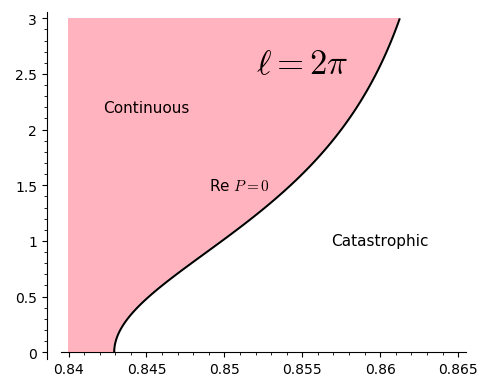}
    \caption{The phase diagram at $\ell = 2\pi$.}
    \label{fig:2pi_phase}
\end{figure}
Therefore we have a continuous transition if $\eta < 0$, so we have a continuous transition if 
\[4b^2 + \frac{2b^2}{9 + 4 \sigma^2} - 3 < 0\]
and a catastrophic transition for the opposite inequality (see \cref{fig:2pi_phase}). This implies that if $b^2 < \frac{27}{38}$, the transition will be continuous, if $b^2 > \frac{3}{4}$, the transition will necessarily be a catastrophic transition. Note that this indeed matches the results of \cref{section:complex_simple}.

We can see the change in the transition type across a critical value of $\sigma$. Take $b = 0.86$, then the critical value is at $\sigma \approx 2.577$. So at $\sigma = 2.6$, we expect to see a continuous transition, with attracting periodic orbits for $\lambda > 0$. And indeed, at $\lambda = 0.01$, we see in the left side graph of \cref{fig:periodic_orbit} that the forward in time solutions (up to $t = 100$) with initial values at $(0.9, 0)$ and $(1.8, 0)$ converge onto a limit cycle, approximated by $r \approx 1.158$. 

At $\sigma = 2.5$, we expect a catastrophic transition, with repelling periodic orbits for $\lambda < 0$. At $\lambda = -0.1$, we see in the right side graph of \cref{fig:periodic_orbit} that the backward in time solutions with initial values at $(0.9, 0)$ and $(0.8, 0)$ also converge onto a limit cycle, approximated by $r \approx 1.122$. 

\begin{figure}
    \centering
    \includegraphics[scale = 0.65]{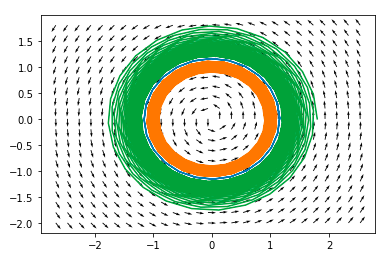}
    \includegraphics[scale = 0.65]{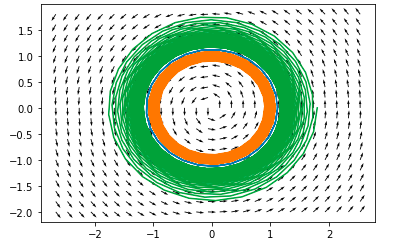}
    \caption{Forward in time trajectories (left) tending towards the stable periodic orbit (blue), and backward in time trajectories (right) tending towards the unstable periodic orbit (blue).}
    \label{fig:periodic_orbit}
\end{figure}

For continuous transitions, we can also numerically approximate the radius of the limit cycles for small values of lambda by simply performing a binary search. As $\lambda \to 0$, the radius should decrease as the square root of lambda with coefficient determined by $\re P$ as in \cref{th:single_hopf}. And indeed, we can see this behavior clearly at $\sigma = 6$ and $b = 0.86$ in \cref{fig:bif_rad}.
\begin{figure}
    \centering
    \includegraphics{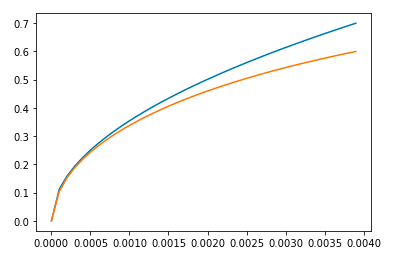}
    \caption{The blue line is the numerical approximations of the radius of the limit cycles as a function of $\lambda$, and the orange line is the analytical limiting behavior as $\lambda \to 0$.}
    \label{fig:bif_rad}
\end{figure}

The special case of $\ell = 2 \pi$ came with a lot of computational conveniences. However, it is easy to see that the dynamics for other $\ell \in \mathcal I_2$ follow similar dynamics, and the line $\re P = 0$ in $\sigma$ - $b$ phase space takes the same shape. Lastly, we give the real center manifold formulas for $\ell \in \mathcal I_2$ and $\mathcal I_4$ in real coordinates. Recall from \cref{complex_center_func_2} that the center manifold function for $\ell \in \mathcal I_2$ is given by
\begin{align*}
    \Phi(u, \lambda) =& \frac{b}{2 \beta_k - \beta_{2k}} z^2 \phi_k^2 + \frac{b}{\ol{2 \beta_k - \beta_{2k}}} \ol{z^2 \phi_k}^2 + \frac{2b}{\beta_k + \ol \beta_k - \beta_0} z \phi_k \ol{z\phi_k} + o ( \| u \|^2) \\
    =& 2 \re \left(\frac{b}{[\lambda - (1 - \rho^2)^2] - (6 \rho^2 - 15\rho^4) + 6 \sigma \rho^3 i} z^2 \phi_k^2 \right) \\
    &+ \frac{2b}{2[\lambda - (1 - \rho^2)^2] + 1 - \lambda} |z|^2 + o(\|u\|^2).
\end{align*}

Then employing the same change of basis with
\[e_1 = \cos(kx), \qquad e_2 = \sin(kx)\]
so that $u = u_1 e_1 + u_2 + e_2$ and $z = \frac{1}{2} (u_1 - i u_2)$, we find the following formula:
\begin{proposition}\label{real_2}
The center manifold function for $\ell \in \mathcal{I}_2$ for \cref{SH_system} is given by 
\[\Phi(u_1, u_2, \lambda)(x) = A(u_1, u_2, \lambda) \cos(2kx) + B(u_1, u_2, \lambda) \sin(2kx) + C(u_1, u_2, \lambda)\]
where
\begin{align*}
    A(u_1, u_2, \lambda) =& \frac{1}{2} \frac{b}{\left([\lambda - (1 - \rho^2)^2] - (15\rho^4 - 6 \rho^2)\right)^2 + \left(6 \sigma \rho^3 \right)^2} \\
    & \left[ \left([\lambda - (1 - \rho^2)^2] - (15\rho^4 - 6 \rho^2)\right) (u_1^2 - u_2^2) - 12u_1 u_2 \sigma \rho^3) \right]
\end{align*}
\begin{align*}
    B(u_1, u_2, \lambda) =& \frac{1}{2} \frac{b}{\left([\lambda - (1 - \rho^2)^2] - (15\rho^4 - 6 \rho^2)\right)^2 + \left(6 \sigma \rho^3 \right)^2} \\
    & \left[\left([\lambda - (1 - \rho^2)^2] - (15\rho^4 - 6 \rho^2)\right)2 u_1 u_2 + 6 \sigma \rho^3 (u_1^2 - u_2^2) \right]
\end{align*}
\[C(u_1, u_2, \lambda) = \frac{1}{2} \frac{b}{2[\lambda - (1 - \rho^2)^2] + 1 - \lambda} (u_1^2 + u_2^2)\]
\end{proposition}

Now for $\ell \in \mathcal I_4$. Then the center manifold function in the complex eigenbasis was given by \cref{4cross_center_func}. We want to change basis to 
\begin{equation*}
    e_1 = \cos(kx) \qquad e_2 = \sin(kx) \qquad e_3 = \cos((k + 1)x) \qquad e_4 = \sin((k + 1)x).
\end{equation*}
We can write $u = \sum_{i = 1}^4 u_i e_i$, so that
\begin{equation*}
    z_1 = \frac{1}{2} (u_1 - i u_2) \qquad z_2 = \frac{1}{2} (u_3 - i u_4)
\end{equation*}
Let 
\begin{gather*}
    c_1 + d_1 i \coloneqq \frac{b}{2 \beta_k - \beta_{2k}} \\
    c_2 + d_2 i \coloneqq \frac{b}{2 \beta_{k + 1} - \beta_{2k + 2}} \\
    c_3 \coloneqq \frac{b}{\beta_k + \ol \beta_k - \beta_0} \\
    c_4 \coloneqq \frac{b}{\beta_{k + 1} + \ol \beta_{k + 1} - \beta_0} \\
    c_5 + d_5 i \coloneqq \frac{2b}{\ol \beta_k + \beta_{k + 1} - \beta_1}.
\end{gather*}
Note that $c_i$ and $d_i$ depend on $\lambda$ for some fixed $\ell$ and $\sigma$. Substituting into the center manifold function, 
\begin{multline*}
    \Phi(u_1, u_2, u_3, u_4, \lambda) = 2 \re \left[ \frac{1}{4} (c_1 + i d_1)(u_1 - i u_2)^2 (\cos(2 k x) + i \sin(2kx))  \right. \\ 
    + \frac{1}{4} (c_2 + i d_2) ( u_3 - i u_4)^2 (\cos ((2k + 2)x) + i \sin((2k + 2)x)) \\
    + \frac{1}{4} c_3 (u_1^2 + u_2^2) + \frac{1}{4} c_4 (u_3^2 + u_4^2) \\
    \left. + \frac{1}{4}(c_5 + i d_5) (u_1 + i u_2) (u_3 - i u_4)(\cos x + i \sin x) \right]
\end{multline*}
This can be rearranged to yield the following:

\begin{proposition}\label{real_4}
The center manifold function for $\ell \in \mathcal I_4$ is given by 
\begin{multline*}
    \Phi(u, \lambda) = A_1 \cos(2kx) + A_2 \sin(2 k x) + B_1 \cos((2k + 2) x) \\ + B_2 \sin((2k + 2)x) + C + D_1 \cos x + D_2 \sin(x) 
\end{multline*}
where
\begin{gather*}
    A_1 = c_1(u_1^2 - u_2^2) + 2d_1 u_1 u_2 \qquad A_2 = 2 c_1 u_1 u_2 - d_1(u_1^2 - u_2^2) \\
    B_1 = c_2(u_3^2 - u_4^2) + 2d_2 u_3 u_4 \qquad B_2 = 2 c_2 u_3 u_4 - d_2(u_3^2 - u_4^2) \\
    C = \frac{1}{4} c_3 (u_1^2 + u_2^2) + \frac{1}{4} c_4 (u_3^2 + u_4^2) \\
    D_1 = c_5(u_1 u_3 + u_2 u_4) + d_5(u_1 u_4 - u_2 u_3) \\
    D_2 = c_5(u_1 u_4 - u_2 u_3) - d_5 (u_1 u_3 + u_2 u_4)
\end{gather*}
\end{proposition}
Thus \cref{real_2} and \cref{real_2} gives the real manifolds on which the exchange of stability occurs for all $\ell > 2\pi/\sqrt{2}$.

\section*{Acknowledgements}
This research was funded by NSF / DMS grant 1757857 as part of the 2020 Indiana Research Experiences for Undergraduates (REU) Program. The author greatly thanks Shouhong Wang for not only suggesting the problem, but also providing resources and support over the weeks spent on the problem. He would also like to thank Dylan Thurston for running the REU program despite all the challenges presented by Covid-19 outbreak.

\medskip

\bibliographystyle{plain}
\bibliography{ptbib.bib}

\end{document}